\documentclass[11pt,reqno]{amsart}
\usepackage[pagewise]{lineno}

\usepackage[brazilian,english]{babel} 
\usepackage[T1]{fontenc}
\usepackage{amsfonts,amsthm,amssymb,amsmath}  
\usepackage{graphicx,color}       
\usepackage{latexsym}       
\usepackage{overpic}        
\usepackage[colorlinks=false, linktocpage=true]{hyperref}
\usepackage{graphicx}
\usepackage{multicol}
\usepackage{color}
\usepackage{multirow} 

\newtheorem{teore}{Theorem}[section]
\newtheorem{obs}[teore]{Remark}
\newtheorem{defi}[teore]{Definition}

\newtheorem{coro}[teore]{Corollary}

\newtheorem{pro}[teore]{Proposition}

\newtheorem{lem}[teore]{Lemma}

\newtheorem{theorem}{}

\newcommand{\E}{E_{\mathbf{0}}}
\newcommand{\ka}{\kappa}
\newcommand{\R}{\mathbb{R}}

\newcommand{\C}{\mathbb{C}}

\newcommand{\psib}{\mbox{\boldmath$\psi$}}

\newcommand{\ub}{\mathbf{u}}
\newcommand{\vb}{\mathbf{v}}
\newcommand{\zb}{\mathbf{z}}

\newcommand{\Lb}{\mathbf{L}}
\numberwithin{equation}{section}

\title[Scattering for NLS systems]{Scattering of radial solutions for quadratic-type Schr\"{o}dinger systems  in dimension five}



\begin{document}

\maketitle	

\begin{center}

{\bf Norman Noguera}

{SM-UCR, Ciudad Universitaria Carlos Monge Alfaro, Departamento de Ciencias Naturales, Apdo: 111-4250, San Ram\'on, Alajuela, Costa Rica.}\\
{email: norman.noguera@ucr.ac.cr}

\vspace{3mm}

{\bf Ademir Pastor}

{IMECC-UNICAMP,
	Rua S\'ergio Buarque de Holanda, 651, Cidade Universit\'aria, 13083-859, Campinas, S\~ao Paulo,
	Brazil.}\\
{email: apastor@ime.unicamp.br}

\end{center}

\begin{abstract}
	In this paper we study the scattering of radial solutions  to a  $l$-component system of nonlinear Schr\"{o}dinger equations with quadratic-type growth interactions in dimension five. Our approach is based on the recent technique introduced by Dodson and Murphy, which relies on the radial Sobolev embedding and a Morawetz estimate. 
\end{abstract}

\tableofcontents

\section{Introduction}

The study of nonlinear Schr\"odinger systems with quadratic nonlinearities has attracted the attention of physicists and mathematicians in recent years. These models appear, for instance, in optical communications when the optical material has quadratic nonlinear response (see e.g. \cite{Colin2} and \cite{kivshar2000multi}). One example is the system
\begin{equation}\label{system1J}
\begin{cases}
\displaystyle i\partial_{t}u_1+\Delta u_1=-2\overline{u}_1u_2,\\
\displaystyle i\partial_{t}u_2+\kappa\Delta u_2=- u^{2}_1,
\end{cases}
\end{equation}
where $u_{1}$ and $u_{2}$ are complex-valued functions on the variables $(x,t)\in\R^n\times\R$, $\ka$ is a real constant and $\Delta$ stands for the usual Laplacian operator. System \eqref{system1J} was introduced in \cite{Hayashi} as a non-relativistic version  of some Klein-Gordon systems, but it can also be derived as a model in nonlinear optics (see \cite{Colin2}). 

The local  well-posedness   of \eqref{system1J}, in the spaces $L^{2}(\R^n)$ and $H^{1}(\R^n)$, was studied in  \cite{Hayashi} for $1\leq n\leq 4$ and $1\leq n\leq 6$, respectively. Since \eqref{system1J} is $L^2$-critical in dimension $n=4$, global well-posedness was also obtained in  $L^{2}(\R^n)$ and  $H^{1}(\R^n)$, for $1\leq n\leq 3$, due to the conservation of the mass and energy. For initial data in $H^{1}(\R^{4})$, a sharp threshold for global well-posedness was proved.    Existence of ground state solutions was also proved in dimensions $1\leq n\leq 6$.  The dichotomy global existence versus blow-up in finite time in $H^1(\R^{5}$) was discussed in \cite{hamano2018global} and \cite{NoPa}. 

Motivated by \eqref{system1J}, in \cite{NoPa2} we initiated the study of Schr\"odinger systems with general quadratic-type nonlinearities. More precisely, we considered the following initial-value problem 
\begin{equation}\label{system1}
\begin{cases}
\displaystyle i\alpha_{k}\partial_{t}u_{k}+\gamma_{k}\Delta u_{k}-\beta_{k} u_{k}=-f_{k}(u_{1},\ldots,u_{l}),\\
(u_{1}(x,0),\ldots,u_{l}(x,0))=(u_{10},\ldots,u_{l0}),\qquad k=1,\ldots l,
\end{cases}
\end{equation}
where $u_{k}:\R^{n}\times \R\to \C$, $(x,t)\in \R^{n}\times \R$, $\alpha_{k}, \gamma_{k}>0$, $\beta_{k}\geq0$ are real constants and  the nonlinearities $f_{k}$ satisfy a quadratic-type growth. To be more specific, we assumed

\renewcommand\thetheorem{(H1)}
\begin{theorem}\label{H1}
 \begin{align*}
f_{k}(0,\ldots,0)=0, \qquad  k=1,\ldots,l. 
\end{align*}
\end{theorem}

\renewcommand\thetheorem{(H2)}
\begin{theorem}\label{H2}
 There exists a constant $C>0$ such that for $(z_{1},\ldots,z_{l}),(z_{1}',\ldots,z_{l}')\in \C^{l}$ we have 
\begin{equation*}
\begin{split}
\left|\frac{\partial }{\partial z_{m}}[f_{k}(z_{1},\ldots,z_{l})-f_{k}(z_{1}',\ldots,z_{l}')]\right|&\leq C\sum_{j=1}^{l}|z_{j}-z_{j}'|,\qquad k,m=1,\ldots,l;\\
\left|\frac{\partial }{\partial \overline{z}_{m}}[f_{k}(z_{1},\ldots,z_{l})-f_{k}(z_{1}',\ldots,z_{l}')]\right|&\leq C\sum_{j=1}^{l}|z_{j}-z_{j}'|,\qquad k,m=1,\ldots,l.
\end{split}
\end{equation*}
\end{theorem}

\renewcommand\thetheorem{(H3)}
\begin{theorem}\label{H3}
There exists a function $F:\C^{l}\to \C$,  such that    
\begin{equation*}
f_{k}(z_{1},\ldots,z_{l})=\frac{\partial F}{\partial \overline{z}_{k}}(z_{1},\ldots,z_{l})+\overline{\frac{\partial F }{\partial z_{k}}}(z_{1},\ldots,z_{l}),\qquad k=1\ldots,l. 
\end{equation*}
\end{theorem}

\renewcommand\thetheorem{(H4)}
\begin{theorem}\label{H4}
 For any $ \theta \in \R$ and $(z_{1},\ldots,z_{l})\in \mathbb{C}^{l}$,
\begin{equation*}
\mathrm{Re}\,F\left(e^{i\frac{\alpha_{1}}{\gamma_{1}}\theta  }z_{1},\ldots,e^{i\frac{\alpha_{l}}{\gamma_{l}}\theta  }z_{l}\right)=\mathrm{Re}\,F(z_{1},\ldots,z_{l}).
\end{equation*}	
\end{theorem}

\renewcommand\thetheorem{(H5)}
\begin{theorem}\label{H5}
Function $F$ is homogeneous of degree 3, that is, for any $\lambda >0$  and $(z_{1},\ldots,z_{l})\in \mathbb{C}^{l}$,
\begin{equation*}
F(\lambda z_{1},\ldots,\lambda z_{l})=\lambda^{3}F(z_{1},\ldots,z_{l}).
\end{equation*}
\end{theorem}

\renewcommand\thetheorem{(H6)}
\begin{theorem}\label{H6}
There holds
\begin{equation*}
\left|\mathrm{Re}\int_{\R^{n}} F(u_{1},\ldots,u_{l})\;dx\right|\leq \int_{\R^{n}} F(|u_{1}|,\ldots,|u_{l}|)\;dx.  
\end{equation*}
\end{theorem}

\renewcommand\thetheorem{(H7)}
\begin{theorem}\label{H7}
Function $F$ is real valued on $\R^l$, that is, if $(y_{1},\ldots,y_{l})\in \R^{l}$ then
\begin{equation*}
F(y_{1},\ldots,y_{l})\in \R. 
\end{equation*}
Moreover, functions	$f_k$ are non-negative on the positive cone in $\mathbb{R}^l$, that is, for $y_i\geq0$, $i=1,\ldots,l$,
\begin{equation*}
f_{k}(y_{1},\ldots,y_{l})\geq0.
\end{equation*}

\end{theorem}


\renewcommand\thetheorem{(H8)}
\begin{theorem}\label{H8}
	Function $F$ can be written as the sum $F=F_1+\cdots+F_m$, where $F_s$, $s=1,\ldots, m$ is super-modular on $\R^d_+$, $1\leq d\leq l$ and vanishes on hyperplanes, that is, for any $i,j\in\{1,\ldots,d\}$, $i\neq j$ and $k,h>0$, we have
	\begin{equation*}
	F_s(y+he_i+ke_j)+F_s(y)\geq F_s(y+he_i)+F_s(y+ke_j), \qquad y\in \R^d_+,
	\end{equation*}
	and $F_s(y_1,\ldots,y_d)=0$ if $y_j=0$ for some $j\in\{1,\ldots,d\}$.
\end{theorem}

\begin{obs}
It is not difficult to check that system \eqref{system1J} satisfies \textnormal{\ref{H1}-\ref{H8}}. In this case we have
\begin{equation*}\label{Fsys}
f_{1}(z_1,z_2)=2\overline{z}_1z_2, \quad f_{2}(z_1,z_2)=z_1^2, \quad \mathrm{and}\quad F(z_1,z_2)=\overline{z}_1^2z_2.
\end{equation*}
For additional models with quadratic nonlinearities satisfying \textnormal{\ref{H1}-\ref{H8}} we refer the reader to \cite{kivshar2000multi}, \cite{NoPa2}, and \cite{Pastor2}.

\end{obs}

The present work is the third of a series of papers concerned with the initial value problem \eqref{system1} endowed  with assumptions \ref{H1}-\ref{H8}. Let us recall some of the results we have established (see Section \ref{sec.prel} for some notations). In \cite{NoPa2} we studied some aspects of the dynamics of \eqref{system1} such as local and global well-posedness, existence of standing waves, the dichotomy global existence versus blow-up in finite time and the stability/instability of standing waves. Since functions $f_{k}$ are homogeneous of degree two,  system \eqref{system1} (with $\beta_{k}=0$) is invariant under the scaling $ u_{k}^{\lambda}(x,t)=\lambda^{2}u_{k}( \lambda x,\lambda^{2} t)$, $k=1,\ldots,l.$. Thus,  by a standard scaling argument it is possible to show  that  $\dot{H}^{n/2-2}(\R^n)$ is the critical Sobolev space. In particular, $L^{2}(\R^n)$ and $\dot{H}^{1}(\R^n)$ are critical in dimensions $n=4$ and $n=6$, respectively. Hence, \eqref{system1} is $L^{2}$-subcritical if $1\leq n\leq 3$ and $H^{1}$-subcritical if $1\leq n\leq 5$. Consequently, using a standard contraction argument based on the Strichartz estimates,  assumptions \ref{H1} and \ref{H2} are enough to show that  \eqref{system1} is locally well-posed  in $L^{2}(\R^{n})$ if $1\leq n\leq 4$  and in $H^{1}(\R^{n})$ if $1\leq n\leq 6$. Next, assuming \ref{H3} and \ref{H4}  it is possible to establish the conservation of the quantities
 	\begin{equation}\label{mass}
	Q(\ub(t)):=\sum_{k=1}^{l}\frac{\alpha_{k}^{2}}{\gamma_{k}}\|  u_{k}(t)\|_{L^{2}}^{2},
	\end{equation}
	and
	\begin{equation}
\label{energy}
	E_{\boldsymbol{\beta}}(\ub(t)):=\sum_{k=1}^{l}\gamma_{k}\|\nabla u_{k}(t)\|_{L^2}^{2}+\sum_{k=1}^{l}\beta_{k}\|u_{k}(t)\|_{L^2}^{2}
    -2\mathrm{Re}\int F(\ub(t))\;dx, 
\end{equation}	
with $\boldsymbol{\beta}=(\beta_{1},\ldots,\beta_{l})$. This means that, as long as a solution exists, it satisfies
\begin{equation}\label{conserQE}
Q(\ub(t))=Q(\ub_{0}) \qquad\mathrm{and}\qquad
E_{\boldsymbol{\beta}}(\ub(t))=E_{\boldsymbol{\beta}}(\ub_{0}).
\end{equation}

\begin{obs}
Here we use  $E_{\boldsymbol{\beta}}$ to indicate the dependence on the parameter $\boldsymbol{\beta}$.  In particular, when $\beta_{k}=0$, $k=1,\ldots,l$ we write
\begin{equation}
\label{energybet0}
E_{\boldsymbol{0}}(\ub(t))=\sum_{k=1}^{l}\gamma_{k}\|\nabla u_{k}(t)\|_{L^2}^{2}
-2\mathrm{Re}\int F(\ub(t))\;dx, 
\end{equation}
\end{obs}

Using these conserved quantities and \ref{H6} we then got  an \textit{a priori} bound for the $L^{2}$ and $H^{1}$-norm of a solution, so the global well-posedness may be established  in $L^{2}(\R^{n})$ and $H^{1}(\R^{n})$, when $1\leq n\leq 3$. In dimensions  $n=4$ and $n=5$ global solutions in $H^1(\R^n)$ may be obtained depending on the size of the initial data compared to that of the ground states associated with \eqref{system1}. To be more precise, recall that a standing wave for \eqref{system1} is a  solution of the form
\begin{equation}\label{standing}
u_{k}(x,t)=e^{i\frac{\alpha_{k}}{\gamma_{k}}\omega t}\psi_{k}(x),\qquad k=1,\ldots,l,
\end{equation}
where $\omega\in \R$ and  $\psi_{k}$ are real-valued functions decaying to zero at infinity, which  
satisfy the following semilinear elliptic system
\begin{equation}\label{systemelip}
\displaystyle -\gamma_{k}\Delta \psi_{k}+\left(\frac{\alpha_{k}^{2}}{\gamma_{k}}\omega+\beta_{k}\right) \psi_{k}=f_{k}(\psib),\qquad k=1,\ldots,l.
\end{equation}
A ground state solution for \eqref{systemelip} is a solution that minimizes the action functional
\begin{equation*}\label{FunctionalI}
I(\boldsymbol{\psi})=\frac{1}{2}\left[\sum_{k=1}^{l}\gamma_{k}\|\nabla \psi_{k}\|_{L^2}^{2}+\sum_{k=1}^{l}\left(\frac{\alpha_{k}^{2}}{\gamma_{k}}\omega+\beta_{k}\right)\| \psi_{k}\|_{L^2}^{2}\right]
-\int F(\boldsymbol{\psi})\;dx.
\end{equation*}
Let  $\mathcal{G}_{n}(\omega,\boldsymbol{\beta})$ denote  the set of ground state solutions of \eqref{systemelip}. Under our assumption, in \cite[Theorem 4.12]{NoPa2} it was shown that if the coefficients $\frac{\alpha_{k}^{2}}{\gamma_{k}}\omega+\beta_{k}$ are positive then $\mathcal{G}_{n}(\omega,\boldsymbol{\beta})\neq \emptyset$ in the $H^{1}$-subcritical case. As a byproduct of this result, introducing the functionals 
\begin{equation}\label{functionalQ}
\mathcal{Q}(\boldsymbol{\psi})=\sum_{k=1}^{l}\left(\frac{\alpha_{k}^{2}}{\gamma_{k}}\omega+\beta_{k}\right)\| \psi_{k}\|_{L^2}^{2},
\end{equation}
\begin{equation}\label{funclKP6}
K(\boldsymbol{\psi})=\sum_{k=1}^{l}\gamma_{k}\|\nabla \psi_{k}\|_{L^2}^{2}\qquad \mathrm{and}\qquad P(\boldsymbol{\psi})=\int F(\boldsymbol{\psi})\;dx,
\end{equation}
we have the following Gagliardo-Nirenberg-type inequality,
\begin{equation}\label{GNI}
    P(\ub)\leq C_{n}^{opt}\mathcal{Q}(\ub)^{\frac{6-n}{4}}K(\ub)^{\frac{n}{4}},
\end{equation}
for all functions $\ub\in \mathcal{P}:=\{\psib\in \mathbf{H}^{1}(\R^n);\, P(\psib)>0\}$, with the optimal constant $C_{n}^{opt}$ given by
\begin{equation}\label{bestCn}
    C_{n}^{opt}:=\frac{2(6-n)^{\frac{n-4}{4}}}{n^{\frac{n}{4}}}\frac{1}{\mathcal{Q}(\psib)^{\frac{1}{2}}},
\end{equation}
where $\psib$ is any function in $\mathcal{G}_{n}(\omega,\boldsymbol{\beta})$. 

\begin{obs}\label{groundrel}
	Let $\boldsymbol{\psi}$ be a solution of system \eqref{systemelip}. Then,
	\begin{equation*}
	P(\boldsymbol{\psi})=2I(\boldsymbol{\psi}), \quad K(\boldsymbol{\psi})=nI(\boldsymbol{\psi}), \quad \mbox{and}\quad \mathcal{Q}(\boldsymbol{\psi})=(6-n)I(\boldsymbol{\psi})
	\end{equation*}
	In particular, it is easy to see that if $n=5$ and  $\boldsymbol{\psi}\in \mathcal{G}_{5}(1,\mathbf{0})$ then $K(\psib)=5Q(\psib)$ and $	E_{\boldsymbol{0}}(\psib)=Q(\psib)$. Note also that in the case $\omega=1$ and $\boldsymbol{\beta}=\mathbf{0}$ the functional $Q$ and $\mathcal{Q}$ coincides.
\end{obs}

As a consequence of \eqref{GNI} we proved (see \cite[Theorem 5.2]{NoPa2}) that, if $\ub_0\in \mathbf{H}^{1}(\R^4)$ satisfies $Q(\ub_0)<Q(\psib)$, where $\psib$ is any function in $\mathcal{G}_{4}(1,\mathbf{0})$ then the corresponding solution of \eqref{system1} may be extended globally in $\mathbf{H}^{1}(\R^4)$. In dimension $n=5$ we established the following.\\

\noindent {\bf Theorem A.} \textit{Let $n=5$. 
	Assume $\mathbf{u}_{0}\in \mathbf{H}^{1}$ and let $\mathbf{u}$ be the corresponding solution of system \eqref{system1}. Let  $\boldsymbol{\psi}\in \mathcal{G}_{5}(1,\mathbf{0})$ be a ground state. If
	\begin{equation}\label{desEQgs}
	Q(\mathbf{u}_{0})	E_{\boldsymbol{\beta}}(\mathbf{u}_{0})<Q(\boldsymbol{\psi})\E(\boldsymbol{\psi}),
	\end{equation}
	where $\E$ is the energy given in \eqref{energybet0}
	and
	\begin{equation}\label{desQKgs}
	Q(\mathbf{u}_{0})K(\mathbf{u}_{0})<Q(\boldsymbol{\psi})K(\boldsymbol{\psi}),  
	\end{equation}
	then, $\ub$ is global in $\mathbf{H}^{1}$.}\\

Our main purpose in this paper is to prove that if we assume that $\ub_0$ is radial then the solution provided by Theorem A scatters. More precisely, our main result is the following.

\begin{teore}\label{thm:critscatt} Let $n=5$. 
In addition to the assumptions of Theorem A, assume also that $\mathbf{u}_{0}$ is radial. Then, $\ub$ is global and scatters forward and backward in time, that is,  there exist $u^{\pm}_{k}\in H^{1}(\R^{5})$ such that 
	\begin{equation*}
	\lim_{t\to +\infty}\|u_{k}(t)-U_{k}(t)u_{k}^{+}\|_{H^{1}}=0, \qquad \mathrm{for}\quad k=1,\ldots,l
	\end{equation*}
and
	\begin{equation*}
\lim_{t\to -\infty}\|u_{k}(t)-U_{k}(t)u_{k}^{-}\|_{H^{1}}=0, \qquad \mathrm{for}\quad k=1,\ldots,l.
\end{equation*}
\end{teore}

Our idea  to prove Theorem \ref{thm:critscatt} is to apply the recent theory introduced in \cite{dodson2017new}, where the authors have shown the scattering for the standard cubic Schr\"odinger equation in dimension $n=3$. 

Let us recall the scattering results for system \eqref{system1J}. First recall that if $\ka=1/2$ then \eqref{system1J} is said to satisfy the mass-resonance condition. In the critical case $n=4$, scattering in $L^2(\R^4)$ was established in \cite{inui2019scattering} under the condition $Q(\ub_0)<Q(\psib)$, with $\psib\in \mathcal{G}_{4}(1,\mathbf{0})$. More precisely, the authors established that scattering holds for any initial data satisfying  $Q(\ub_0)<Q(\psib)$ in the mass-resonance case and for radial initial data satisfying  $Q(\ub_0)<Q(\psib)$ without the mass-resonance condition. 

In dimension $n=5$, under similar assumption as in Theorem A, the scattering of radially symmetric solutions in $H^1(\R^5)$ was established in \cite{hamano2018global} and \cite{hamano2019scattering} with the assumption of mass-resonance and without the assumption of mass-resonance, respectively. In both cases, the authors used the concentration-compactness and rigidity method introduced in \cite{KenigMarleScattering}. These results were improved in \cite{wang2019Sacttering} and \cite{mengxu2020} where, in the mass-resonance case, using the ideas introduced in \cite{dodsonmurphy2018},  the authors dropped the assumption of radial initial data.

Before ending this introduction let us recall the notion of mass-resonance associated with \eqref{system1}: we say that \eqref{system1} satisfies the mass-resonance condition provided (see \cite[Definition 1.1]{NoPa3})
	\begin{equation}\label{RC}
\mathrm{Im}\sum_{k=1}^{l}\frac{\alpha_{k}}{2\gamma_{k}}f_{k}(\zb)\overline{z}_{k}=0, \quad \zb\in \mathbb{C}^{l}.\tag{RC}
\end{equation}

\begin{obs}
	It is easy to see that \eqref{system1J} satisfies \eqref{RC} if and only if $\ka=1/2$.
\end{obs}

In view of Lemma 2.9 in \cite{NoPa2}, we recall that if \eqref{system1} satisfies \ref{H3} and \ref{H4} then \eqref{RC} holds. So our main result in Theorem \ref{thm:critscatt} is obtained under the mass-resonance condition. We believe that using the ideas in \cite{dodsonmurphy2018} we may omit the assumption of radial symmetry (in the mass-resonance case)  and replace the assumption \ref{H4} by

\renewcommand\thetheorem{(H4*)}
\begin{theorem}\label{H4*}
	There exist positive constants $\sigma_{1},\ldots,\sigma_{l}$ such that for any $\zb\in \mathbb{C}^{l}$
	\begin{equation*}
	\mathrm{Im}\sum_{k=1}^{l}\sigma_{k}f_{k}(\zb)\overline{z}_{k}=0 .
	\end{equation*}
\end{theorem}

This is currently under investigation.

It is to be clear that assumption \ref{H4*} is enough to prove the conservation of the quantity $Q$ defined in \eqref{mass} (see \cite{NoPa3}). Consequently, Theorem A still holds replacing \ref{H4} by \ref{H4*}. In addition, in \cite{NoPa3}, under assumption \ref{H4*}, we have investigated the existence of ground state solutions in the $\dot{H}^1$-critical case as well as the existence of finite time blow-up solutions.

This work is organized as follows. In section \ref{sec.prel} we recall some notation and preliminary lemmas that will be needed throughout the paper.
In section \ref{sec.scattcrit}, using the strategy introduced in \cite{tao2004asymptotic}, we prove a scattering criterion for \eqref{system1} in the radial case. Finally, section \ref{sec.proofscatt} is devoted to showing Theorem \ref{thm:critscatt}, by using a generalized virial/Morawetz identity and the ground state solutions.



\section{Preliminaries}\label{sec.prel}
In this section we introduce some notations, review some useful estimates and give  consequences of our assumptions.
\subsection{Notation}
 We use $C$ to denote several positive constants that may vary line-by-line. If $a$ and $b$ are two positive constants, by $a\lesssim b$ we mean there is a constant $C$ such that $a\leq Cb$.
Given any set $A$, by  $\mathbf{A}$ (or $A^{l}$) we denote  the product  $\displaystyle A\times \cdots \times A $ ($l$ times). In particular, if $A$ is a Banach space then $\mathbf{A}$ is also a Banach space with the standard norm given by the sum.  For a number $z\in\mathbb{C}$, $\mathrm{Re}\,z$  and $\mathrm{Im}\,z$ represents its real and imaginary parts. Also, $\overline{z}$ denotes its complex conjugate.  Given $\mathbf{z}=(z_{1},\ldots,z_{l})\in \mathbb{C}^l$, we write $z_m=x_m+iy_m$ where $x_m=\mathrm{Re}\,z_m$ and $y_m=\mathrm{Im}\,z_m$. The operators $\partial/\partial z_m$ and $\partial/\partial \overline{z}_m$ are defined by
$$
\dfrac{\partial}{\partial z_m}=\frac{1}{2}\left(\frac{\partial}{\partial x_m} -i\frac{\partial}{\partial y_m}\right), \qquad\dfrac{\partial}{\partial \overline{z}_m}=\frac{1}{2}\left(\frac{\partial}{\partial x_m} +i\frac{\partial}{\partial y_m}\right).
$$ 
The space
 $L^{p}=L^{p}(\R^{n})$, $1\leq p\leq \infty$, stands for the standard Lebesgue spaces. By $W^s_p=W^s_p(\R^{n})$, $1\leq p\leq \infty$, $s\in\R$, we denote the usual Sobolev spaces. In the case $p=2$, we use the standard notation $H^s=W^s_2$. Thus, $\mathbf{H}^1=\mathbf{H}^1(\R^n)$ denotes the Sobolev space $H^1\times \cdots \times H^1$.

To simplify notation, if no confusion is caused we use $\int f\, dx$ to denote  $\int_{\R^n} f\, dx$.
Given a time interval $I$, the mixed   spaces $L^{p}_{t}L^{q}_{x}(I\times\R^n)$ are endowed with the norm
$$
\|f\|_{L^{p}_{t}L^{q}_{x}(I\times\R^n)}=\left(\int_I \left(\int_{\R^n}|f(x,t)|^qdx \right)^{\frac{p}{q}} dt \right)^{\frac{1}{p}},
$$
with the obvious modification if either $p=\infty$ or $q=\infty$. If no confusion will be caused we denote  $L^{p}_{t}L^{q}_{x}(I\times\R^n)$ simply by  $L^{p}_{t}L^{q}_{x}$ and its norm by $\|\cdot\|_{L^{p}_{t}L^{q}_{x}}$. Also, when $p=q$ we will use $L^{p}_{tx}$ instead of $L^{p}_{t}L^{q}_{x}$. More generally, if $X$ is a Banach space, $L^{p}(I;X)$ represents the $L^{p}$ space of $X$-valued functions defined on $I$.

\subsection{Some useful estimates} 

Let us start by recalling from Duhamel's principle the Cauchy problem \eqref{system1} can be written as the following  system of integral equations,
\begin{equation}\label{system2}
\begin{cases}
u_{k}(t)= \displaystyle U_{k}(t)u_{k0}+i\int_{0}^{t}U_{k}(t-t') \frac{1}{\alpha_{k}}f_{k}(\mathbf{u})\;dt',\\
(u_{1}(x,0),\ldots,u_{l}(x,0))=(u_{10},\ldots,u_{l0})=:\ub_0,
\end{cases}
\end{equation}
where $U_{k}(t)$ is the Schr\"odinger evolution group defined by $\displaystyle U_{k}(t)=e^{i\frac{t}{\alpha_{k}}(\gamma_{k}\Delta-\beta_{k})}$, $ k=1\ldots, l$.

We have the following dispersive estimate.

\begin{lem}\label{dispest}
	If $2\leq p\leq\infty$ and $t\neq 0$, then $U_{k}(t)$ maps $L_{x}^{p'}(\R^{n})$ continuously to $L_{x}^{p}(\R^{n})$ and
	\begin{equation*}
	\|U_{k}(t)f\|_{L_{x}^{p}(\R^{n})}\lesssim |t|^{-n\left(\frac{1}{2}-\frac{1}{p}\right)}\|f\|_{L_{x}^{p'}(\R^{n})}, \qquad \mbox{for all $f\in L_{x}^{p'}(\R^{n})$.}
	\end{equation*}
\end{lem}
\begin{proof}
	See Proposition 2.2.3 in \cite{Cazenave}.
\end{proof}

Before proceeding we recall the definition of admissible pair in dimension $n=5$.

\begin{defi}
We say that $(q,r)$ is an admissible  pair if
\begin{equation*}
    \frac{2}{q}+\frac{5}{r}=\frac{5}{2},
\end{equation*}
where $2\leq r\leq \frac{10}{3}$.
\end{defi}

We now state  the well known Strichartz inequalities. 

\begin{pro}[Strichartz's inequalities]\label{stricha}
The following inequalities hold.
\begin{itemize}
	\item[(i)] Let $(q,r)$ be an admissible pair. Then,
	$$
	\|U_{k}(t)f\|_{L_{t}^{q}L_{x}^{r}(\R\times\R^{5})}\lesssim\|f\|_{L_{x}^{2}(\R^{5})}
	$$
	\item[(ii)] Let $I$ be an interval  and $t_0\in \overline{I}$. Let $(q_1,r_1)$ and $(q_2,r_2)$ be two admissible pairs. Then, 
	$$
	\left\| \int_{t_0}^tU_{k}(t-s)f(\cdot,s)ds \right\|_{L_{t}^{q_{1}}L_{x}^{r_{1}}(I\times\R^{5})} \lesssim \|f\|_{L_{t}^{q_{2}'}L_{x}^{r_{2}'}(I\times\R^{5})},
	$$
	where $q_2'$ and $r_2'$ are the H\"older conjugates of $q_2$ and $r_2$, respectively.
	\item[(iii)] Let $I$ be an interval  and $t_0\in \overline{I}$. Then
		$$
	\left\| \int_{t_0}^tU_{k}(t-s)f(\cdot,s)ds \right\|_{L_{t}^{6}L_{x}^{3}(I\times\R^{5})} \lesssim \|f\|_{L_{t}^{3}L_{x}^{\frac{3}{2}}(I\times\R^{5})}.
	$$
\end{itemize}
\end{pro}
\begin{proof}
	For (i) and (ii) see, for instance, Theorem 2.3.3 in \cite{Cazenave}. For (iii) note that $(6,3)$ and $(\frac{3}{2},3)$ are $\frac{5}{2}$-acceptable pairs. Recall that a pair $(q,r)$ is $\sigma$-acceptable if
	$$
	\frac{1}{q}<2\sigma\left(\frac{1}{2}-\frac{1}{r}\right).
	$$
	Hence the result follows from Proposition 6.2 in \cite{foschi}.
\end{proof}

In the proof of our main result, we need the well known  Strauss' radial lemma and one of its consequence. 

\begin{lem} Let $f\in H^{1}(\R^{n})$ be radially symmetric and suppose $n\geq 2$. Then,  for any $R>0$,
\begin{equation}\label{strausslem}
\|f\|_{L_{x}^{\infty}(|x|\geq R)}\lesssim R^{-(n-1)/2}\|f\|^{1/2}_{L_{x}^{2}(|x|\geq R)}\|\nabla
f\|^{1/2}_{L_{x}^{2}(|x|\geq R)}
\end{equation} 
and
\begin{equation}\label{strausslem1}
\|f\|_{L_{x}^{p+1}(|x|\geq R)}^{p+1}\lesssim R^{-(n-1)(p-1)/2}\|f\|^{(p+3)/2}_{L_{x}^{2}(|x|\geq R)}\|\nabla
f\|^{(p-1)/2}_{L_{x}^{2}(|x|\geq R)}.
\end{equation}
\end{lem}
\begin{proof}
For \eqref{strausslem} see  \cite[Lemma 1.7.3 ]{Cazenave}. Inequality  \eqref{strausslem1} is a consequence of \eqref{strausslem} (see \cite[equation (3.7)]{Ogawa}).
\end{proof}


Next, we will present some consequences of our assumptions on the non-linear terms.

 \begin{lem}\label{estdiffk} Assume that  \textnormal{\ref{H1}-\ref{H5}}   hold. 
 
 \begin{enumerate}
     \item[(i)] For all $\mathbf{z}\in\mathbb{C}^l$  we have
\begin{equation*}
\left|f_{k}(\mathbf{z})\right|\lesssim \sum_{j=1}^{l}|z_{j}|^{2}, \qquad k=1\ldots, l
\end{equation*}
and
\begin{equation*}
|\mathrm{Re}\,F(\mathbf{z})|\lesssim  \sum_{j=1}^{l}|z_{j}|^{3}.
\end{equation*}
\item[(ii)] We have
$$
\mathrm{Im}\sum_{k=1}^l\frac{\alpha_{k}}{\gamma_{k}}f_k(\mathbf{z})\overline{z}_k=0.
$$
\item[(iii)] Let $1< p,q,r< \infty$ be such that $\frac{1}{r}=\frac{1}{p}+\frac{1}{q}$.  Then, for $k=1,\ldots,l$,
\begin{equation}\label{lei11}
	\| f_{k}(\mathbf{u})\|_{L^{r}}\lesssim \|\mathbf{u}\|_{\mathbf{L}_{x}^{p}}\| \nabla\mathbf{u}\|_{\mathbf{L}_{x}^{q}(\R^n)}
\end{equation}
and 
\begin{equation}\label{lei12}
\|f_{k}(\mathbf{u})\|_{\mathbf{W}^{\frac{1}{2},r}}\lesssim \|\mathbf{u}\|_{\mathbf{L}_{x}^{p}}\| \mathbf{u}\|_{\mathbf{W}_{x}^{\frac{1}{2},q}(\R^n)}
\end{equation}
 \end{enumerate}
\end{lem}
\begin{proof}
For (i) and (ii) see  Corollary 2.3 and Lemmas 2.9 and 2.10 in  \cite{NoPa2}. Part (iii) is a consequence of \ref{H2} and the Leibniz rule (see Proposition 5.1 in \cite{mtaylor} and Corollary 2.5 in \cite{NoPa2}).
\end{proof}

We finish this section with the following coercivity result. 

 \begin{lem}[Coercivity I]\label{thm:lemcoerI} Let $n=5$. 
 	Assume $\mathbf{u}_{0}\in \mathbf{H}_{x}^{1}$  and let $\mathbf{u}$ be the corresponding solution of system \eqref{system1} with maximal existence interval $I$. Let  $\boldsymbol{\psi}\in \mathcal{G}_{5}(1,\mathbf{0})$ be a ground state. If 
 \begin{equation}\label{b-1}
   Q(\mathbf{u}_{0})	E_{\boldsymbol{\beta}}(\mathbf{u}_{0})<(1-\delta)Q(\boldsymbol{\psi})\E(\boldsymbol{\psi})
\end{equation}
and
\begin{equation}\label{b-2}
  Q(\mathbf{u}_{0})K(\mathbf{u}_{0})\leq Q(\boldsymbol{\psi})K(\boldsymbol{\psi}),  
\end{equation}
then there exist $\delta'>0$, depending on $\delta$, so that
\begin{equation*}
  Q(\mathbf{u}_{0})K(\mathbf{u}(t))<(1-\delta') Q(\boldsymbol{\psi})K(\boldsymbol{\psi}),  
\end{equation*}
for all $t\in I$. In particular, $I=\R$ and $\ub$ is uniformly bounded in $\mathbf{H}_{x}^{1}$. 
 \end{lem}
\begin{proof}
First, by the conservation of 	$E_{\boldsymbol{\beta}}$, \ref{H6}, and \eqref{GNI} we have
\begin{equation}\label{b1}
	E_{\boldsymbol{\beta}}(\ub_0)\geq K(\ub(t))-2C_5^{opt}Q(\ub(t))^{\frac{1}{4}}K(\ub(t))^{\frac{5}{4}}.
\end{equation}
Multiplying \eqref{b1} by $Q(\ub(t))=Q(\ub_0)$ and using \eqref{b-1} we obtain
\[
(1-\delta)Q(\boldsymbol{\psi})\E(\boldsymbol{\psi})>K(\ub(t))Q(\ub(t))-2C_5^{opt}Q(\ub(t))^{\frac{5}{4}}K(\ub(t))^{\frac{5}{4}},
\]
or, equivalently,
\begin{equation}\label{b2}
1-\delta>5\frac{K(\ub(t))Q(\ub(t))}{K(\psib)Q(\psib)}-10C_5^{opt}\frac{Q(\ub(t))^{\frac{5}{4}}K(\ub(t))^{\frac{5}{4}}}{K(\psib)Q(\psib)},
\end{equation}
where we used that $K(\psib)=5\E(\boldsymbol{\psi})$ (see Remark \ref{groundrel}). Now using \eqref{bestCn} (with $n=5$) and that $K(\psib)=5Q(\psib)$ it it easy to see that
\[
10C_5^{opt}\frac{1}{K(\psib)Q(\psib)}=\frac{4}{K(\psib)^{\frac{5}{4}}Q(\psib)^{\frac{5}{4}}}.
\]
Thus, from \eqref{b2},
\[
1-\delta>5\frac{K(\ub(t))Q(\ub(t))}{K(\psib)Q(\psib)}-4\left(\frac{K(\ub(t))Q(\ub(t))}{K(\psib)Q(\psib)}\right)^{\frac{5}{4}}.
\]
Since, from \eqref{b-2} we have $\frac{K(\ub_0)Q(\ub_0)}{K(\psib)Q(\psib)}\leq 1$, a simple continuity argument shows that $\frac{K(\ub(t))Q(\ub(t))}{K(\psib)Q(\psib)}<\delta'$, for some $\delta'>0$, which completes the proof of the lemma.
\end{proof}

 \section{Scattering criterion}\label{sec.scattcrit}
The following result is an adapted version of Theorem 1.1 in \cite{tao2004asymptotic}. See also \cite[Lemma 3.2]{wang2019Sacttering} \cite[Lemma 2.2]{dodson2017new} and \cite[Lemma 2.5]{arora2019scattering}. We will be concerned only with scattering forward in time; in a similar fashion we may also prove the scattering backward in time.

\begin{lem}[Scattering criterion]\label{scatcrit}
Suppose that $\ub$ is a radially-symmetric solution of \eqref{system1} satisfying 
\begin{equation}\label{unifbond}
    \|\ub\|_{\mathbf{L}_{t}^{\infty}(\mathbf{H}_{x}^{1})}\leq A.
\end{equation}
There exist positive constants $\epsilon$ and $R$, depending on $A$, such that if
\begin{equation}\label{liminfl2}
    \liminf_{t\to \infty}\int_{|x|\leq R}\sum_{k=1}^{l}\frac{\alpha_k^{2}}{\gamma_{k}}|u_{k}(x,t)|^{2}\;dx\leq \epsilon^{2},
\end{equation}
then $\ub$ scatters forward in time. 
\end{lem}
\begin{proof}
First of all note that $(\frac{7}{2},\frac{70}{27})$ is an admissible pair. Hence, by Proposition \ref{stricha} we have that
$ \|U_{k}(t)u_{k0}\|_{L_{t}^{\frac{7}{2}}W_{x}^{\frac{1}{2},\frac{70}{27}}(\R\times \R^{5})}$  is finite for any $k=1,\ldots,l$. Thus, from the Sobolev embedding (see for instance \cite[Theorem 1.3.3]{Cazenave}) 
\begin{equation}\label{sobemb}
   W_{x}^{\frac{1}{2},\frac{70}{27}}(\R^{5})\hookrightarrow L_{x}^{\frac{7}{2}}(\R^{5}), 
\end{equation}
 we obtain that $\|U_{k}(t)u_{k0}\|_{L_{tx}^{\frac{7}{2}}(\R\times \R^{5})}$ is also finite for any $k=1,\ldots,l$.
Consequently, for $T_{0}$ large enough (depending on $\ub_{0}$) we get
\begin{equation}\label{bondT0}
    \sum_{k=1}^{l}\|U_{k}(t)u_{k0}\|_{L_{tx}^{\frac{7}{2}}([T_{0},\infty)\times \R^{5})}<\epsilon,
\end{equation}
where  $0<\epsilon<1$ is a small constant to be chosen later.
 Now  let $\chi\in C_0^\infty(\R^5)$ be a radial smooth function such that $0\leq \chi\leq1$ and
\begin{equation}\label{defchi}
\chi(x)=\left\{\begin{array}{cc}
1,& |x|\leq \frac{1}{2},\\
0,& |x|\geq 1.
\end{array}\right.
\end{equation}
For any $R>0$ we set $\chi_{R}(x)=\chi(x/R)$. The parameter $R$  will also be chosen later.

 By hypothesis  \eqref{liminfl2}, there exists $T>T_{0}$ such that
\begin{equation}\label{Tepsilon}
    \int   \sum_{k=1}^{l}\frac{\alpha_k^{2}}{\gamma_{k}}\chi_{R}(x)|u_{k}(x,T)|^{2}\;dx\leq \epsilon^{2}.
\end{equation}
Also, using \eqref{system1} and Lemma \ref{estdiffk}-(ii) we obtain the identity 
$$\sum_{k=1}^{l}\frac{\alpha_k^{2}}{\gamma_{k}}\chi_{R}(x)\partial_{t}(|u_{k}|^{2})=-2\mathrm{Im}\sum_{k=1}^{l}\alpha_{k}\chi_{R}(x)\overline{u}_{k}\Delta u_{k},$$
which after integrating on $\R^5$ and using integration by parts yield
\begin{equation*}
    \begin{split}
        \partial_{t}\left[\int\sum_{k=1}^{l}\frac{\alpha_k^{2}}{\gamma_{k}}\chi_{R}(x)|u_{k}|^{2}\;dx\right]= 2\mathrm{Im}\int\sum_{k=1}^{l}\alpha_{k}\nabla \chi_{R}(x)\overline{u}_{k}\nabla u_{k}\;dx.
    \end{split}
\end{equation*}
Since $|\nabla \chi_{R}(x)|\leq \frac{C}{R}$, the last identity and Young's inequality give
\begin{equation}\label{dertunfr}
    \begin{split}
        \left|\partial_{t}\left[\int\sum_{k=1}^{l}\frac{\alpha_k^{2}}{\gamma_{k}}\chi_{R}(x)|u_{k}|^{2}\;dx\right]\right|
        \leq \frac{C}{R}\int\sum_{k=1}^{l}\alpha_{k}[|u_{k}|^{2}+|\nabla u_{k}|^{2}]\;dx
        \leq  \frac{C}{R}\|\ub\|^{2}_{\mathbf{H}_{x}^{1}}
        \leq \frac{C}{R},
    \end{split}
\end{equation}
where we have used \eqref{unifbond}.

Let $\alpha>0$ be small to be chosen later and define $I_{1}:=[T-\epsilon^{-\alpha},T]$. Choose $R$ sufficiently large such that $R\geq\epsilon^{-\alpha-2}$. By using  \eqref{dertunfr} we obtain for any $t\in I_1$,
\begin{equation*}
\begin{split}
    \int \sum_{k=1}^{l}\frac{\alpha_k^{2}}{\gamma_{k}}\chi_{R}(x)|u_{k}(x,t)|^{2}\;dx&\leq \int\sum_{k=1}^{l}\frac{\alpha_k^{2}}{\gamma_{k}}\chi_{R}(x)|u_{k}(x,T)|^{2}\;dx+\int_{T-\epsilon^{-\alpha}}^{T}\frac{C}{R}\;dx\\
    &\lesssim \epsilon^{2}+\frac{\epsilon^{-\alpha}}{R}\\
    &\lesssim \epsilon^{2},
    \end{split}
\end{equation*}
where we have used \eqref{Tepsilon} and the choice of $R$. Hence,
\begin{equation}\label{bondL2nor}
    \|\chi_{R}\ub\|_{\mathbf{L}^{\infty}_{t}\mathbf{L}^{2}_{x}(I_{1}\times \R^{5})}\lesssim \epsilon. 
\end{equation}

Now, by using interpolation, we get 
\begin{equation*}
    \begin{split}
    \|u_{k}(t)\|_{L^{\infty}_{t}L^{\frac{21}{8}}_{x}(I_{1}\times  \R^{5})} &\leq   \|\chi_Ru_{k}(t)\|_{L^{\infty}_{t}L^{\frac{21}{8}}_{x}(I_{1}\times \R^5)}+ \|(1-\chi_R)u_{k}(t)\|_{L^{\infty}_{t}L^{\frac{21}{8}}_{x}(I_{1}\times  \R^{5})}\\
    &\leq \|\chi_Ru_{k}(t)\|_{L^{\infty}_{t}L^{2}_{x}(I_{1}\times \R^{5})}^{\frac{17}{42}}\|u_{k}(t)\|_{L^{\infty}_{t}L^{\frac{10}{3}}_{x}(I_{1}\times \R^{5})}^{\frac{25}{42}}\\
    &\quad + \|(1-\chi_{R})u_{k}(t)\|_{L^{\infty}_{tx}(I_{1}\times \R^{5})}^{\frac{5}{21}}\|u_{k}(t)\|_{L^{\infty}_{t}L^{2}_{x}(I_{1}\times \R^{5})}^{\frac{16}{21}}.
    \end{split}
\end{equation*}
Note that the last term and the term with the norm $L^{\infty}_{t}L^{\frac{10}{3}}_{x}$ on the right-hand side of the last inequality may be bounded in view of  \eqref{unifbond} and Sobolev's embedding $H_{x}^{1}(\R^{5})\hookrightarrow L_{x}^{\frac{10}{3}}(\R^{5})$. For the two remaining terms we may  use \eqref{bondL2nor} and Lemma  \ref{strausslem} to obtain
\begin{equation*}
   \|u_{k}(t)\|_{L^{\infty}_{t}L^{\frac{21}{8}}_{x}(I_{1}\times  \R^{5})}\lesssim \epsilon^{\frac{17}{42}}+R^{-\frac{10}{21}}. 
\end{equation*}
Hence for $R$ sufficiently large (depending on $\epsilon$) we have
\begin{equation}\label{bondu1}
     \|\ub\|_{\mathbf{L}^{\infty}_{t}\mathbf{L}^{\frac{21}{8}}_{x}(I_{1}\times  \R^{5})}\lesssim \epsilon^{\frac{17}{42}}. 
\end{equation}

Now, by using Duhamel's formula (see \eqref{system2}) we write
\begin{equation}\label{duham1}
    \begin{split}
        U_{k}(t-T)u_{k}(T)&=U_{k}(t)u_{k0}+i\int_{0}^{T}U_{k}(t-s)\frac{1}{\alpha_{k}}f_{k}(\ub)(s)\;ds.
    \end{split}
\end{equation}

\noindent {\bf Claim 1.} $ \|U_{k}(t-T)u_{k}(T)\|_{L^{\frac{7}{2}}_{tx}([T,\infty)\times \R^{5})}\lesssim \epsilon^\beta$, form some $0<\beta<1$.\\

To prove the claim let us first set $I_{2}=[0,T-\epsilon^{-\alpha}]$ and
$$
\mathcal{F}_{j}(t):=i\int_{I_{j}}U_{k}(t-s)\frac{1}{\alpha_{k}}f_{k}(\ub)(s)\;ds,\quad j=1,2.
$$ 
Thus, from \eqref{duham1} we can express
\begin{equation}\label{duham2}
   U_{k}(t-T)u_{k}(T)=  U_{k}(t)u_{k0}+\mathcal{F}_{1}(t)+\mathcal{F}_{2}(t)
\end{equation}
and obtain
\begin{equation}\label{sumIj}
\begin{split}
    \|U_{k}(t-T)u_{k}(T)\|_{L^{\frac{7}{2}}_{tx}([T,\infty)\times \R^{5})}&\leq  \|U_{k}(t)u_{k0}\|_{L^{\frac{7}{2}}_{tx}([T,\infty)\times \R^{5})}+\|\mathcal{F}_{1}(t)\|_{L^{\frac{7}{2}}_{tx}([T,\infty)\times \R^{5})}\\
    &\quad+\|\mathcal{F}_{2}(t)\|_{L^{\frac{7}{2}}_{tx}([T,\infty)\times \R^{5})}\\
    &=:\mathcal{I}_{0}+\mathcal{I}_{1}+\mathcal{I}_{2}.
\end{split}
\end{equation}
To bound $\mathcal{I}_{0}$, we use that $T>T_{0}$ and \eqref{bondT0}, to get
\begin{equation}\label{I1bound}
    \mathcal{I}_{0}\leq \|U_{k}(t)u_{k0}\|_{L^{\frac{7}{2}}_{tx}([T_{0},\infty)\times \R^{5})}<\epsilon.
\end{equation}

For $\mathcal{I}_{1}$, interpolation and the  Sobolev embedding $W_{x}^{\frac{1}{2},\frac{20}{7}}(\R^5)\hookrightarrow L_{x}^4(\R^5)$ give
\begin{equation*}
\begin{split}
\mathcal{I}_{1}&\leq \left\|\mathcal{F}_1\right\|_{L_{t}^{6}L_{x}^{3}([T,\infty)\times \R^{5})}^{\frac{3}{7}}  \left\|\mathcal{F}_1\right\|_{L_{t}^{\frac{8}{3}}L_{x}^{4}([T,\infty)\times \R^{5})}^{\frac{4}{7}}\\
&\lesssim \left\|\mathcal{F}_1\right\|_{L_{t}^{6}L_{x}^{3}([T,\infty)\times \R^{5})}^{\frac{3}{7}}  \left\|\mathcal{F}_1\right\|_{L_{t}^{\frac{8}{3}}W_{x}^{\frac{1}{2},\frac{20}{7}}([T,\infty)\times \R^{5})}^{\frac{4}{7}}.
\end{split}
\end{equation*}
Thus, since $\left(\frac{8}{3},\frac{20}{7}\right)$ and $\left(\frac{7}{3},\frac{70}{23}\right)$ are admissible pairs, from Proposition \ref{stricha} we deduce
\begin{equation}\label{aa1}
\mathcal{I}_{1}\lesssim \left\|f_k(\mathbf{u})\right\|_{L_{t}^{3}L_{x}^{\frac{3}{2}}(I_1\times \R^{5})}^{\frac{3}{7}}  \left\|f_k(\mathbf{u})\right\|_{L_{t}^{\frac{7}{4}}W_{x}^{\frac{1}{2},\frac{70}{47}}(I_1\times \R^{5})}^{\frac{4}{7}}
\end{equation}
Now, from Lemma \ref{estdiffk}-(i), H\"older's inequality and \eqref{sobemb}, we have
\[
\begin{split}
\left\|f_k(\mathbf{u})\right\|_{L_{t}^{3}L_{x}^{\frac{3}{2}}(I_1\times \R^{5})}& \lesssim
\|\mathbf{u}\|_{\mathbf{L}^{\frac{7}{2}}_{tx}(I_1\times \R^{5})} \left\|\mathbf{u}\right\|_{\mathbf{L}_{t}^{21}\mathbf{L}_{x}^{\frac{21}{8}}(I_1\times \R^{5})}\\
& \lesssim \left\|\mathbf{u}\right\|_{\mathbf{L}_{t}^{\frac{7}{2}}\mathbf{W}_{x}^{\frac{1}{2},\frac{70}{27}}(I_1\times \R^{5})} \left\|\mathbf{u}\right\|_{\mathbf{L}_{t}^{21}\mathbf{L}_{x}^{\frac{21}{8}}(I_1\times \R^{5})}.
\end{split}
\]
Also, from Lemma \ref{estdiffk}-(iii), \eqref{sobemb}  and H\"older's inequality,
\[
\begin{split}
\left\|f_k(\mathbf{u})\right\|_{\mathbf{L}_{t}^{\frac{7}{4}}\mathbf{W}_{x}^{\frac{1}{2},\frac{70}{47}}(I_1\times \R^{5})}& \lesssim \|\mathbf{u}\|_{\mathbf{L}_{t}^{\frac{7}{2}}\mathbf{L}_{x}^{\frac{7}{2}}(I_1\times \R^{5})} \|\mathbf{u}\|_{\mathbf{L}_{t}^{\frac{7}{2}}\mathbf{W}_{x}^{\frac{1}{2},\frac{70}{27}}(I_1\times \R^{5})}\\
& \lesssim \|\mathbf{u}\|_{\mathbf{L}_{t}^{\frac{7}{2}}\mathbf{W}_{x}^{\frac{1}{2},\frac{70}{27}}(I_1\times \R^{5})}^2.
\end{split}\]
Hence, from \eqref{aa1},
\[
\mathcal{I}_{1}\lesssim \left\|\mathbf{u}\right\|_{\mathbf{L}_{t}^{\frac{7}{2}}\mathbf{W}_{x}^{\frac{1}{2},\frac{70}{27}}(I_1\times \R^{5})}^{\frac{11}{7}} \left\|\mathbf{u}\right\|_{\mathbf{L}_{t}^{21}\mathbf{L}_{x}^{\frac{21}{8}}(I_1\times \R^{5})}^\frac{3}{7}.
\]

Next, since $\left(\frac{7}{2},\frac{70}{27}\right)$ is admissible, a continuity argument (see, for instance, \cite[Lemma 3.2]{tao2004asymptotic}) gives that
\begin{equation*}
\left\|\mathbf{u}\right\|_{\mathbf{L}_{t}^{\frac{7}{2}}\mathbf{W}_{x}^{\frac{1}{2},\frac{70}{27}}(I_1\times \R^{5})}\lesssim (1+|I_1|)^{\frac{2}{7}}.
\end{equation*}
As a consequence, H\"older's inequality and \eqref{bondu1} yield
\begin{equation}\label{I2bound}
\begin{split}
\mathcal{I}_{1}&\lesssim (1+|I_1|)^{\frac{22}{49}}|I_1|^{\frac{3}{147}} \left\|\mathbf{u}\right\|_{\mathbf{L}_{t}^{\infty}\mathbf{L}_{x}^{\frac{21}{8}}(I_1\times \R^{5})}^\frac{3}{7}\\
&\lesssim (1+\epsilon^{-\alpha})^{\frac{22}{49}+\frac{3}{147}}\epsilon^\frac{51}{294}\\
&\lesssim \epsilon^\beta,
\end{split}
\end{equation}
provided $\alpha$ is chosen to be small enough.

Finally, for  $\mathcal{I}_{2}$, we first use interpolation to obtain
\begin{equation}\label{I3interp}
    \begin{split}
     \mathcal{I}_{2}\leq \left\|\mathcal{F}_{2}(t)\right\|_{L_{tx}^{\frac{14}{5}}([T,\infty)\times \R^{5})}^{\frac{1}{2}}\left\|\mathcal{F}_{2}(t)\right\|_{L_{tx}^{\frac{14}{3}}([T,\infty)\times \R^{5})}^{\frac{1}{2}}.
    \end{split}
\end{equation}
Next, we will bound each term in the last inequality. Recalling the definition of $\mathcal{F}_{2}(t)$ and using   Duhamel's   formula we have
\begin{equation*}
    \begin{split}
        U_{k}(t-T+\epsilon^{-\alpha})u_{k}(T-\epsilon^{-\alpha})
        &=U_{k}(t)u_{k0}+\mathcal{F}_{2}(t),
    \end{split}
\end{equation*}
or, equivalently,
\begin{equation*}
    \mathcal{F}_{2}(t)=U_{k}(t)\left[U_{k}(-T+\epsilon^{-\alpha})u_{k}(T-\epsilon^{-\alpha})-u_{k0}\right].
\end{equation*}
Since $\left(\frac{14}{5},\frac{14}{5}\right)$ is an admissible pair, the Strichartz estimates and \eqref{mass} give
\begin{equation}\label{boundL145}
    \begin{split}
        \left\|\mathcal{F}_{2}(t)\right\|_{L_{tx}^{\frac{14}{5}}([T,\infty)\times \R^{5})}
        &\lesssim  \left\|U_{k}(-T+\epsilon^{-\alpha})u_{k}(T-\epsilon^{-\alpha})-u_{k0}\right\|_{L_{x}^{2}}
        \lesssim\|\ub_0\|_{\mathbf{L}^{2}_{x}}\leq C.
    \end{split}
\end{equation}

On the other hand,  by  combining Lemma \ref{dispest} with Lemma \ref{estdiffk}-(i), H\"{o}lder inequality  and the Sobolev embedding $H^{1}_{x}\hookrightarrow L^{\frac{5}{2}}_{x}$ and $H^{1}_{x}\hookrightarrow L^{\frac{70}{27}}_{x}$, we have for all $t\in [T,\infty)$,
\begin{equation*}
    \begin{split}
     \left\|\mathcal{F}_{2}(t)\right\|_{L_{x}^{\frac{14}{3}}}
     &\lesssim \int_{0}^{T-\epsilon^{-\alpha}}\left\|U_{k}(t-s)f_{k}(\ub)(s)\right\|_{L_{x}^{\frac{14}{3}}}\;ds\\
     &\lesssim \sum_{j=1}^{l}\int_{0}^{T-\epsilon^{-\alpha}}|t-s|^{-\frac{10}{7}}\|u_{j}^{2}(s)\|_{L^{\frac{14}{11}}_{x}}\;ds\\
     &\lesssim \sum_{j=1}^{l}\int_{0}^{T-\epsilon^{-\alpha}}|t-s|^{-\frac{10}{7}}\|u_{j}(s)\|_{L^{\frac{5}{2}}_{x}}\|u_{j}(s)\|_{L^{\frac{70}{27}}_{x}}\;ds\\
     &\lesssim \sum_{j=1}^{l}\int_{0}^{T-\epsilon^{-\alpha}}|t-s|^{-\frac{10}{7}}\|u_{j}(s)\|_{H^{1}_{x}}^{2}\;ds\\
    \end{split}
\end{equation*}
After using  \eqref{unifbond} and integrating  we obtain
\begin{equation*}
   \left\|\mathcal{F}_{2}(t)\right\|_{L_{x}^{\frac{14}{3}}}\lesssim \left(t-T+\epsilon^{-\alpha}\right)^{-\frac{3}{7}},
\end{equation*}
implying that 
\begin{equation}\label{boundL143}
   \left\|\mathcal{F}_{2}(t)\right\|_{L_{tx}^{\frac{14}{3}}([T,\infty)\times \R^{5})}\lesssim \epsilon^{\frac{3}{14}\alpha}.
\end{equation}
Putting together  \eqref{I3interp}, \eqref{boundL145} and \eqref{boundL143} we get
\begin{equation}\label{I3bound}
    \mathcal{I}_{2}\lesssim\epsilon^{\frac{3\alpha}{28}}.
\end{equation}

As a consequence of \eqref{sumIj}, \eqref{I1bound}, \eqref{I2bound}  and \eqref{I3bound}  the claim is proved.\\

Next we will show that Claim 1 implies the scattering. First note from Duhamel's formula that for $k=1, \ldots,l$
\begin{equation}\label{duhasol}
u_{k}(t)= \displaystyle U_{k}(t-T)u_{k}(T)+i\int_{T}^{t}U_{k}(t-s) \frac{1}{\alpha_{k}}f_{k}(\mathbf{u})\;ds.
\end{equation}

\noindent {\bf Claim 2.} $\|\ub\|_{\Lb^6_t\Lb^3_x([T,\infty)\times\R^5)}\lesssim \epsilon^{\frac{7\beta}{12}}$.\\

Indeed,  \eqref{duhasol}, interpolation, Claim 1, Strichartz and H\"older's inequalities, and Sobolev's embedding give
\[
\begin{split}
\|u_k\|_{L^6_tL^3_x}& \lesssim \|U_k(t-T)u_k(T)\|_{L_{tx}^{\frac{7}{2}}}^{\frac{7}{12}} 
\|U_k(t-T)u_k(T)\|_{L_{t}^{\infty}L_x^{\frac{5}{2}}}^{\frac{5}{12}}+\|f_k(\ub)\|_{L_t^3L_x^{\frac{3}{2}}}\\
&\lesssim \epsilon^{\frac{7\beta}{12}}\|U_k(t-T)u_k(T)\|_{L_{t}^{\infty}H_x^{1}}^{\frac{5}{12}}+ \|f_k(\ub)\|_{L_t^3L_x^{\frac{3}{2}}},
\end{split}
\]
where all space-time norms are on $[T,\infty)\times\R^5$. By using Lemma \ref{estdiffk}-(i), H\"older's inequality and \eqref{unifbond} we then obtain
\[
\begin{split}
	\|u_k\|_{L^6_tL^3_x}& \lesssim  \epsilon^{\frac{7\beta}{12}}\|u_k(T)\|_{H_x^{1}}^{\frac{5}{12}}+ \sum_{j=1}^l\||u_j|^2\|_{L_t^3L_x^{\frac{3}{2}}}\\
	&\lesssim \epsilon^{\frac{7\beta}{12}}+\|\ub\|_{\Lb^6_t\Lb^3_x}^2.
\end{split}
\]
By summing over $k$ and using a continuity argument the claim is proved.\\

\noindent {\bf Claim 3.} $\|\ub\|_{\Lb_t^{\frac{12}{5}} \mathbf{W}_x^{1,3}([T,\infty)\times\R^5)}$ is finite.

In fact, since $\left(\frac{12}{5},3\right)$ is admissible, Strichartz's inequality yields
\[
\begin{split}
\|u_k\|_{L_t^{\frac{12}{5}} {W}_x^{1,3}}\lesssim \|u_{k0}\|_{H^{1}_{x}}+\|f_k(\ub)\|_{L_t^{\frac{12}{7}}W_x^{1,\frac{3}{2}}},
\end{split}
\]
where again all space-time norms are on $[T,\infty)\times\R^5$. In view of Lemma \ref{estdiffk}-(iii) and H\"older's inequality we then deduce
\begin{equation}\label{aa2}
\|u_k\|_{L_t^{\frac{12}{5}} {W}_x^{1,3}}\lesssim \|u_{k0}\|_{H^{1}_{x}}+ \|\ub\|_{\Lb_t^6\Lb_x^3}\|\ub\|_{\Lb_t^{\frac{12}{5}}\mathbf{W}_x^{1,3}}.
\end{equation}
Since, from Claim 2, $\|\ub\|_{\Lb_t^6\Lb_x^3}$ is sufficiently small (by choosing $\epsilon$ small) another continuity argument gives the desired.\\

With Claims 1, 2 and 3 in hand we are able to prove the scattering. Indeed, since $(\infty,2)$ is admissible, a similar argument as in Claim 3 shows that the integral
$$
\int_{T}^{+\infty}U_{k}(-s) \frac{1}{\alpha_{k}}f_{k}(\mathbf{u})\;ds, 
$$
converges in $H^1(\R^5)$, for $k=1,\ldots,l$. Thus we may define
\begin{equation*}
     u_{k}^{+}:=U_{k}(-T)u_{k}(T)+i\int_{T}^{+\infty}U_{k}(-s) \frac{1}{\alpha_{k}}f_{k}(\mathbf{u})\;ds,
 \end{equation*}
 which in turn belongs to $H_{x}^1(\R^5)$. A straightforward calculation gives,
  for $k=1,\ldots,l$ and $t\geq T$
 \begin{equation*}
     \begin{split}
      u_{k}(t)-U_{k}(t)u_{k}^{+}=i\int_{t}^{\infty}U_{k}(t-s)\frac{1}{\alpha_{k}}f_{k}(\ub(s))\;ds.
     \end{split}
 \end{equation*}
Hence,  using that $(\infty,2)$ is admissible, as in \eqref{aa2} we obtain
 \begin{equation*}
     \|u(t)-U_{k}(t)u_{k}^{+}\|_{H^{1}_{x}(\R^{5})}\lesssim \|\ub\|_{\Lb_t^6\Lb_x^3}\|\ub\|_{\Lb_t^{\frac{12}{5}}\mathbf{W}_x^{1,3}},
 \end{equation*}
 where the space-time norms are over $[t,+\infty)\times\R^5$. Taking the limit as $t\to+\infty$  in last inequality, the right-hand side goes to zero in view of Claims 2 and 3.
 This finishes the proof. 
\end{proof}

 \section{Proof of the main result}\label{sec.proofscatt}
 This section is devoted to prove Theorem \ref{thm:critscatt}. The main ingredient is a virial/Morawetz-type estimate. To prove it we first state some useful lemmas based on the ground state solutions. Throughout the section we assume the assumptions  in Theorem \ref{thm:critscatt} and let $\chi_{R}$ be the radial cut-off function introduced in Section \ref{sec.scattcrit}.

 \subsection{Coercivity lemmas} Here we will prove two coercivity lemmas.

 \begin{lem}[Coercivity II]\label{lemcoerII} Suppose $Q(\mathbf{v})K(\mathbf{v})<(1-\delta)Q(\boldsymbol{\psi})K(\boldsymbol{\psi})$. Then, there exists $\delta'>0$, depending on $\delta$, so that
\begin{equation*}
  K(\mathbf{v})- \frac{5}{2}\mathrm{Re}\int F(\vb)\;dx\geq \delta'\int\sum_{k=1}^{l}|v_{k}|^{3}\;dx.
\end{equation*}
 \end{lem}
 \begin{proof}
 Using \eqref{energybet0} we  write
 \begin{equation}\label{relKPE}
    K(\mathbf{v})- \frac{5}{2}\mathrm{Re}\int F(\vb)\;dx=\frac{5}{4}\E(\vb)-\frac{1}{4}K(\psib). 
 \end{equation}
Since $K(\psib)=5Q(\psib)$ (see Remark \ref{groundrel}),  the best constant \eqref{bestCn} can be expressed as
\begin{equation}\label{C5}
    C_{5}^{opt}=\frac{2}{5}\frac{1}{Q(\psib)^{\frac{1}{4}}K(\psib)^{\frac{1}{4}}}. 
\end{equation}
 The last identity combined with the Gagliardo-Nirenberg inequality \eqref{GNI} then yields
 \begin{equation*}
     \begin{split}
         \E(\vb) &\geq   K(\mathbf{v})-2C_{5}^{opt}Q(\vb)^{\frac{1}{4}}K(\vb)^{\frac{5}{4}}\\
         &\geq  K(\mathbf{v})\left[1-2C_{5}^{opt}Q(\vb)^{\frac{1}{4}}K(\vb)^{\frac{1}{4}}\right]\\
         &>K(\mathbf{v})\left[1-2C_{5}^{opt}(1-\delta)^{\frac{1}{4}}Q(\psib)^{\frac{1}{4}}K(\psib)^{\frac{1}{4}}\right]\\
         &=K(\mathbf{v})\left[1-\frac{4}{5}(1-\delta)^{\frac{1}{4}}\right],
         \end{split}
 \end{equation*}
 with $1-\frac{4}{5}(1-\delta)^{\frac{1}{4}}>0$. Thus, from \eqref{relKPE} we obtain
 \begin{equation}\label{KPKdes}
   K(\mathbf{v})- \frac{5}{2}\mathrm{Re}\int F(\vb)\;dx\geq \frac{5}{4}\left[1-\frac{4}{5}(1-\delta)^{\frac{1}{4}}\right]K(\mathbf{v})-\frac{1}{4}K(\mathbf{v})=\left[1-(1-\delta)^{\frac{1}{4}}\right] K(\mathbf{v}),
 \end{equation}
 where $1-(1-\delta)^{\frac{1}{4}}>0$.
 
On the other hand,  the standard Gagliardo-Nirenberg inequality and \eqref{C5} imply
\begin{equation*}
\begin{split}
 \int\sum_{k=1}^{l}|v_{k}|^{3}\;dx\leq CQ(\vb)^{\frac{1}{4}}K(\vb)^{\frac{5}{4}}
  &<C(1-\delta)^{\frac{1}{4}}\left[Q(\psib)^{\frac{1}{4}}K(\psib)^{\frac{1}{4}}\right]K(\vb)\\
  &=\frac{2}{5}\frac{C}{C_{5}^{opt}}(1-\delta)^{\frac{1}{4}}K(\vb). 
\end{split}
\end{equation*}
This combined with \eqref{KPKdes} gives the desired with $\delta'=\delta'(\delta)=\frac{5C_{5}^{opt}}{2C}\frac{\left[1-(1-\delta)^{\frac{1}{4}}\right]}{(1-\delta)^{\frac{1}{4}}}>0$. Thus, the proof of the lemma is completed. 
 \end{proof}
 
 \begin{obs}\label{remarksup}
 	Under the assumptions of Theorem \ref{thm:critscatt}, an application of Lemma \ref{thm:lemcoerI} yields we may choose
  $\delta>0$ such that $\sup_{t\in \R}Q(\ub_{0})K(\ub(t))<(1-2\delta)Q(\psib)K(\psib)$. 
 \end{obs}

 \begin{lem}[Coercivity on balls]\label{lemcoerbal} Let $\delta>0$ be as in Remark \ref{remarksup}. Then, there exists $R=R(\delta,Q(\ub),\psib)>0$ sufficiently large such that 
 \begin{equation*}
     \sup_{t\in \R}Q(\chi_{R}\ub(t))K(\chi_{R}\ub(t))<(1-\delta)Q(\psib)K(\psib).
     \end{equation*}
    In particular, by Lemma \ref{lemcoerII}, there exists $\delta'=\delta'(\delta)>0$ so that
    \begin{equation*}
        K(\chi_{R}\ub(t))-\frac{5}{2}\mathrm{Re}\int F(\chi_{R}\ub(t))\;dx \geq \delta'\int\sum_{k=1}^{l}|\chi_{R}u_{k}(t)|^{3}\;dx,
    \end{equation*}
    uniformly for $t\in \R$. 
 \end{lem}
 \begin{proof}
First we observe that  integrating by parts and using that $\chi_{R}$ has  compact support, we obtain
 \begin{equation}\label{intequal}
    \int \chi_{R}^{2} |\nabla u_{k}|^{2}\; dx=\int  |\nabla (\chi_{R} u_{k})|^{2}\;dx +\int \chi_{R}\Delta (\chi_{R})|u_{k}|^{2}\;dx. 
 \end{equation}
Multiplying by $\gamma_{k}$ and summing over $k$ we have
 \begin{equation}\label{identK}
     \int \chi_{R}^{2}\sum_{k=1}^{l} \gamma_{k}|\nabla u_{k}|^{2}\; dx= K(\chi_{R} \ub) +\int \chi_{R}\Delta (\chi_{R})\sum_{k=1}^{l}\gamma_{k}|u_{k}|^{2}\;dx.
 \end{equation}
But since  $|\Delta \chi_{R}(x)|\leq \frac{C}{R^{2}}$, we infer 

\begin{equation}\label{b3}
    \begin{split}
        K(\chi_{R}\ub) &= \int \chi_{R}^{2}\sum_{k=1}^{l} \gamma_{k}|\nabla u_{k}|^{2}\; dx -\int \chi_{R}\Delta (\chi_{R})\sum_{k=1}^{l}\gamma_{k}|u_{k}|^{2}\;dx\\
        &\leq K(\ub)+\frac{C}{R^{2}}Q(\ub_{0}). 
    \end{split}
\end{equation}
Thus by noting that  $Q(\chi_{R}\ub(t))\leq Q(\ub(t))=Q(\ub_{0})$, from \eqref{b3} and Remark \ref{remarksup}  we have 
\begin{equation*}
    \begin{split}
        Q(\chi_{R}\ub(t))K(\chi_{R}\ub(t))&\leq Q(\ub_{0})K(\chi_{R}\ub(t))\\
        &\leq Q(\ub_{0})\left[K(\ub(t))+\frac{C}{R^{2}}Q(\ub_{0})\right]\\
        &\leq Q(\ub_{0})K(\ub(t))+\frac{C}{R^{2}}Q(\ub_{0})^{2}\\
        &<(1-2\delta)Q(\psib)K(\psib)+\frac{C}{R^{2}}Q(\ub_{0})^{2}.
    \end{split}
\end{equation*}
Choosing $R$ sufficiently large  such that $\frac{C}{R^{2}}Q(\ub_{0})^{2}<\delta
Q(\psib)K(\psib)$, we obtain 
\begin{equation*}
    \sup_{t\in \R}Q(\chi_{R}\ub(t))K(\chi_{R}\ub(t))<(1-\delta)Q(\psib)K(\psib).
\end{equation*}
The result is thus established.
 \end{proof}
 
 \subsection{A Morawetz estimate}
 
 Next we will prove a Morawetz estimate adapted to system \eqref{system1}. First we recall the following.

 \begin{lem}[Morawetz identity]\label{Virialinde} Assume $ \mathbf{u}_0 \in \mathbf{H}_{x}^{1}$ and  let $\mathbf{u}$ be the corresponding solution of \eqref{system1}. Assume $\varphi \in C^{\infty}(\R^{5})$ and define
\begin{equation*}
M(t)=\sum_{k=1}^{l}\alpha_{k}\mathrm{Im}\int\nabla \varphi\cdot \nabla u_{k} \overline{u}_{k}\;dx.
\end{equation*}
Then,
\begin{equation}\label{secondervgeralcase}
\begin{split}
M'(t)&=2\sum_{1\leq m,j\leq 5}\mathrm{Re}\int\frac{\partial^{2}\varphi}{\partial {x_{m}}\partial {x_{j}}}\left[\sum_{k=1}^{l}\gamma_{k}\partial_{x_{j}}\overline{u}_{k}\partial_{x_{m}}u_{k}\right]dx\\
&\quad-\frac{1}{2}\int\Delta^{2}\varphi\left(\sum_{k=1}^{l}\gamma_{k}|u_{k}|^{2}\right)\;dx-\mathrm{Re}\int\Delta\varphi F\left(\mathbf{u}\right)\;dx.
\end{split}
\end{equation}
\end{lem}
\begin{proof}
See Theorem 5.7 in \cite{NoPa2}.
\end{proof}

\begin{coro}
Under the assumptions of Lemma \ref{Virialinde}, if $\varphi(x)=\Phi(|x|)$ and $\mathbf{u}_0$ are 
 radially symmetric functions,  we can write \eqref{secondervgeralcase} as
 \begin{equation}\label{secondervradialcase}
 \begin{split}
M'(t)&=2\int \Phi''\left(\sum_{k=1}^{l}\gamma_{k}|\nabla u_{k}|^{2}\right)dx-\frac{1}{2}\int\Delta^{2}\varphi\left(\sum_{k=1}^{l}\gamma_{k}|u_{k}|^{2}\right)dx\\
&\quad-\mathrm{Re}\int\Delta\varphi\, F\left(\mathbf{u}\right)\;dx,
\end{split}
\end{equation}
where the prime represents the derivative with respect to $r=|x|$.
 \end{coro}
\begin{proof}
See Corollary 5.8 in \cite{NoPa2}.
\end{proof}

Next, we will choose an appropriate localized function. Following the strategy in \cite{dodson2017new}, for  $R\gg 1$ to be chosen later let us introduce a radial function $a(x)$  satisfying
 \begin{equation}\label{defa}
 a(x)=\begin{cases}
|x|^{2},&|x|\leq R,\\
3R|x|,& |x|\geq 2R,
\end{cases}
\end{equation}
In the intermediate region $R<|x|\leq 2R$, we assume that
\begin{equation}\label{intermcond}
    a'\geq 0,\qquad a''\geq 0 \qquad \mathrm{and}\qquad |\partial^{\alpha}a|\lesssim R|x|^{-|\alpha|+1}, \quad |\alpha|\geq 1,
\end{equation}
where (with abuse of notation) if $r=|x|$ then $a'(r)=\partial_r a:= \nabla a \cdot\frac{x}{r}$. Note that
\begin{equation*}
    \partial_{x_{j}}a=a'(r)\frac{x_{j}}{r}\qquad \mathrm{and} \qquad \frac{\partial^{2}a}{\partial_{x_{j}}\partial_{x_{m}}}=a''(r)\frac{x_{j}}{r}\frac{x_{m}}{r}+a'(r)\left(\frac{\delta_{jm}}{r}-\frac{x_{j}x_{m}}{r^{3}}\right).
\end{equation*}
In particular, $|\nabla a|\lesssim R$. Under this conditions we have the following:
\begin{enumerate}
    \item[i)] For $|x|\leq R$, we have $a(x)=r^2$. Thus,
    \begin{equation}\label{balintR}
        a''=2,\qquad \Delta a=10 \qquad \mathrm{and} \qquad \Delta^{2}a=0. 
    \end{equation}
    \item[ii)] If $|x|>2R$, we have $a(x)=3Rr$. So,
    \begin{equation}\label{baloutR}
       a''=0, \quad \Delta a=\frac{12R}{r} \quad \mathrm{and}\quad \Delta^{2}a=-\frac{24R}{r^{3}}.
    \end{equation}
\end{enumerate}

\begin{pro}[Virial/Morawetz estimate]\label{viri-Morz} Let $T>0$ be given. For $R=R(\delta,Q(\ub_0),\psib)$ sufficiently large we have
\begin{equation}\label{virmorineq}
    \frac{1}{T}\int_{0}^{T}\int_{|x|\leq R}\sum_{k=1}^{l}|u_{k}(t)|^{3}dt\leq C(\ub,\delta)\left(\frac{R}{T}+\frac{1}{R^{2}}\right).
\end{equation}

\end{pro}

\begin{proof}
Let  $R=R(\delta,Q(\ub_0),\psib)$ be as in Lemma \ref{lemcoerbal} and consider $M$ as in Lemma \ref{Virialinde} with $\varphi$ replaced by the function $a$ defined above. Using the Cauchy-Schwarz inequality, we have
\begin{equation*}
    \begin{split}
        |M(t)|\leq \sum_{k=1}^{l}\alpha_{k}\int|\nabla u_{k}||u_{k}||\nabla a|\; dx
        \leq CR\|\ub\|^{2}_{\mathbf{H}_{x}^{1}}.
    \end{split}
\end{equation*}
Hence,  the uniform bound for $\ub $ in \eqref{unifbond} implies
\begin{equation}\label{boundM}
    \sup_{t\in \R}|M(t)|\lesssim R. 
\end{equation}

Next we estimate $M'(t)$. Let us decompose \eqref{secondervradialcase} as the sum $\mathcal{R}_{1}+\mathcal{R}_{2}+\mathcal{R}_{3}$,  where $\mathcal{R}_{1}$, $\mathcal{R}_{2}$, and $\mathcal{R}_{3}$  mean that the integrals are taken on $\{x:|x|\leq R\}$,  $\{x:R<|x|\leq 2R\}$ and  $\{x:|x|> 2R\}$, respectively. Thus, from \eqref{balintR} we obtain
\begin{equation}\label{R1}
    \begin{split}
        \mathcal{R}_{1}= 4\left[\int_{|x|\leq R}\sum_{k=1}^{l}\gamma_{k}|\nabla u_{k}|^{2}\;dx-\frac{5}{2}\int_{|x|\leq R} \mathrm{Re}\,F\left(\mathbf{u}\right)\;dx\right].
    \end{split}
\end{equation}
Also, taking into account  \eqref{intermcond}, for $\mathcal{R}_{2}$ we have 
\begin{equation}\label{R2}
\begin{split}
\mathcal{R}_{2}&\geq 2\int_{R<|x|\leq 2R} a''\left(\sum_{k=1}^{l}\gamma_{k}|\nabla u_{k}|^{2}\right)dx-C\int_{R<|x|\leq 2R}\frac{R}{|x|^{3}}\left(\sum_{k=1}^{l}\gamma_{k}| u_{k}|^{2}\right)\;dx\\
&\quad-C\int_{R<|x|\leq 2R}\frac{R}{|x|} |\mathrm{Re}\,F\left(\mathbf{u}\right)|\;dx\\
&\geq -\frac{C}{R^{2}}Q(\ub)-C\int_{R<|x|\leq 2R}\frac{R}{|x|} |\mathrm{Re}\,F\left(\mathbf{u}\right)|\;dx,
\end{split}
\end{equation}
where we have discard the first term, since $a''\geq 0$.

Finally, using \eqref{baloutR} we obtain
\begin{equation}\label{R3}
\begin{split}
\mathcal{R}_{3}&=12\int_{|x|>2R}\frac{R}{|x|^{3}}\left(\sum_{k=1}^{l}\gamma_{k}|u_{k}|^{2}\right)\;dx-12\int_{|x|>2R}\mathrm{Re}\,  F\left(\mathbf{u}\right)\;dx\\
&\geq -12\int_{|x|>2R}|\mathrm{Re}\,  F\left(\mathbf{u}\right)|\;dx,
\end{split}
\end{equation}
where we have eliminated the non-negative term.  

Putting together \eqref{R1}-\eqref{R3} we obtain
\begin{equation}\label{Mprime1}
\begin{split}
M'(t)&\geq \mathcal{R}_1-\frac{C}{R^{2}}Q(\ub)
 -C\int_{R<|x|\leq 2R}\frac{R}{|x|} |\mathrm{Re}\,F\left(\mathbf{u}\right)|\;dx -12\int_{|x|>2R}|\mathrm{Re}\,  F\left(\mathbf{u}\right)|\;dx\\
 &\geq  \mathcal{R}_1-\frac{C}{R^{2}}Q(\ub)-C\int_{|x|>R/2}|\mathrm{Re}\,F\left(\mathbf{u}\right)|\;dx.
\end{split} 
\end{equation}

Now observe that
\begin{equation*}
    \begin{split}
        \mathcal{R}_1&\geq 4\left[\int_{|x|\leq R}\sum_{k=1}^{l}\chi_{R}^{2}\gamma_{k}|\nabla u_{k}|^{2}\;dx-\frac{5}{2}\int_{|x|\leq R} \mathrm{Re}\,F\left(\mathbf{u}\right)\;dx\right]\\
        &\geq  4\left[\int_{|x|\leq R}\sum_{k=1}^{l}\chi_{R}^{2}\gamma_{k}|\nabla u_{k}|^{2}\;dx-\frac{5}{2}\int_{|x|\leq R} \chi_{R}^{3}\mathrm{Re}\,F\left(\mathbf{u}\right)\;dx\right]\\
        &\quad -10\int_{|x|\leq R}|\chi_{R}^{3}-1||\mathrm{Re}\,F\left(\mathbf{u}\right)|\;dx
        \end{split}
  \end{equation*}
 Thus, using  \eqref{identK} and Lemma \ref{lemcoerbal} we obtain
  \begin{equation*}
  \begin{split}
        \mathcal{R}_1&\geq 4\left[K(\chi_{R} \ub)-\frac{5}{2}\mathrm{Re}\int F(\chi_{R}\ub(t))\;dx\right]+4\int \chi_{R}\Delta (\chi_{R})\sum_{k=1}^{l}\gamma_{k}|u_{k}|^{2}\;dx\\
        &\quad  -10\int_{|x|\leq R}|\chi_{R}^{3}-1||\mathrm{Re}\,F\left(\mathbf{u}\right)|\;dx\\
        &\geq \delta'\int\sum_{k=1}^{l}|\chi_{R}u_{k}(t)|^{3}\;dx-\frac{C}{R^{2}}Q(\ub)-10\int_{|x|\leq R}|\chi_{R}^{3}-1||\mathrm{Re}\,F\left(\mathbf{u}\right)|\;dx.
    \end{split}
\end{equation*}
Since $\chi_R\equiv1$ if $|x|\leq \frac{R}{2}$, for the last term in the above inequality we have 
\begin{equation*}
\begin{split}
  -10\int_{|x|\leq R}|\chi_{R}^{3}-1||\mathrm{Re}\,F\left(\mathbf{u}\right)|\;dx
  &= -10   \int_{R/2<|x|\leq R}|\chi_{R}^{3}-1||\mathrm{Re}\,F\left(\mathbf{u}\right)|\;dx\\
  &\geq -20 \int_{|x|> R/2}|\mathrm{Re}\,F\left(\mathbf{u}\right)|\;dx.  
\end{split}
\end{equation*}
Hence, 
\begin{equation}\label{ineqI}
   \mathcal{R}_1\geq \delta'\int\sum_{k=1}^{l}|\chi_{R}u_{k}|^{3}\;dx-\frac{C}{R^{2}}Q(\ub)-20 \int_{|x|> R/2}|\mathrm{Re}\,F\left(\mathbf{u}\right)|\;dx.
\end{equation}
Using \eqref{Mprime1} and \eqref{ineqI} we  conclude
\begin{equation}\label{deltaReF}
  \delta'\int\sum_{k=1}^{l}|\chi_{R}u_{k}(t)|^{3}\;dx\leq M'(t)+ \frac{C}{R^{2}}Q(\ub) +C \int_{|x|> R/2}|\mathrm{Re}\,F\left(\mathbf{u}\right)|\;dx.
\end{equation}
 From   Lemma \ref{estdiffk}-(i), \eqref{strausslem1} and \eqref{unifbond}  we have  
 \begin{equation*}
    \begin{split}
   C \int_{|x|> R/2}|\mathrm{Re}\,F\left(\mathbf{u}\right)|\;dx&\leq C \sum_{k=1}^{l}\|u_{k}\|^{3}_{L_{x}^{3}(|x|\geq R/2)}\\
    &\leq C  \sum_{k=1}^{l}R^{-2}\|u_{k}\|^{\frac{5}{2}}_{L_{x}^{2}(|x|\geq R/2)}\|\nabla u_{k}\|^{\frac{1}{2}}_{L_{x}^{2}(|x|\geq R/2)}\\
    &\leq \frac{C}{R^{2}}Q(\ub_{0})^{\frac{5}{2}}\|\ub\|_{\mathbf{L}_{t}^{\infty}(\mathbf{H}_{x}^{1})}^{\frac{1}{2}}\\
    &\leq C(\ub)\frac{1}{R^{2}}.
    \end{split}
\end{equation*}
 Therefore, \eqref{deltaReF} becomes
 \begin{equation*}
     \delta'\int\sum_{k=1}^{l}|\chi_{R}u_{k}(t)|^{3}\;dx\leq M'(t)+ C(\ub)\frac{1}{R^{2}}.
 \end{equation*}

 Integrating on $[0,T]$, using the fundamental theorem of calculus and \eqref{boundM} we get
 \begin{equation*}
 \begin{split}
   \int_{0}^{T}\delta'\int\sum_{k=1}^{l}|\chi_{R}u_{k}|^{3}\;dxdt&\leq M(T)-M(0)+ C(\ub)\frac{T}{R^{2}}\\
   &\leq C(\ub)\left[ R+ \frac{T}{R^{2}}\right].
 \end{split}
\end{equation*}
 This implies that
 \begin{equation*}
      \frac{1}{T}\int_{0}^{T}\int_{|x|\leq R}\sum_{k=1}^{l}|u_{k}(t)|^{3}\;dxdt\leq \frac{C(\ub)}{\delta'}\left[ \frac{R}{T}+ \frac{1}{R^{2}}\right].
 \end{equation*}
 Hence, \eqref{virmorineq} follows by recalling that $\delta'$ depends on $\delta$.  
\end{proof}

\subsection{Proof of Theorem \ref{thm:critscatt}} Here we will conclude the proof of our main result. The last ingredient is the following energy evacuation result.

\begin{pro}[Energy evacuation]\label{energevac}
There exist  sequences of real numbers $t_{m}\to +\infty$ and $R_{m}\to +\infty$ such that
\begin{equation*}
    \lim_{m\to \infty}\int_{|x|\leq R_{m}}\sum_{k=1}^{l}|u_{k}(x,t_{m})|^{3}\;dx=0. 
\end{equation*}
\end{pro}

\begin{proof}
We may assume $\ub_0\neq\mathbf{0}$, otherwise the result is trivial. Taking $T_{m}\to \infty$ and $R_{m}=T_{m}^{\frac{1}{3}}$ in Proposition  \ref{viri-Morz} we have
\begin{equation}\label{1}
    \frac{1}{T_{m}}\int_{0}^{T_{m}}\int_{|x|\leq R_{m}}\sum_{k=1}^{l}|u_{k}(x,t)|^{3}\;dxdt\leq CT_{m}^{-\frac{2}{3}} \to 0,\quad \mathrm{as} \quad m\to \infty. 
\end{equation}

 Next,  define the function $\mathcal{H}(t):=\int_{0}^{t}\int_{|x|\leq R_{m}}\sum_{k=1}^{l}|u_{k}(x,\tau)|^{3}\;dxd\tau$. Since $\ub$ is a global solution, by the Gagliardo-Nirenberg inequality we have that $\mathcal{H}(t)$ is well defined. For $m$ fixed we have that $\mathcal{H}$ is a continuous function on $[0,T_{m}]$ and differentiable on $(0,T_{m})$. Then, by the mean value theorem there exists $t_{m}\in (0,T_{m})$ such that
\begin{equation}\label{2}
   \mathcal{H}'(t_{m})=\frac{\mathcal{H}(T_{m})-\mathcal{H}(0)}{T_{m}} =\frac{1}{T_{m}}\int_{0}^{T_{m}}\int_{|x|\leq R_{m}}\sum_{k=1}^{l}|u_{k}(x,\tau)|^{3}\;dxd\tau.
\end{equation}

Combining \eqref{1} with \eqref{2} we get
\begin{equation*}
    \int_{|x|\leq R_{m}}\sum_{k=1}^{l}|u_{k}(x,t_{m})|^{3}\;dx=\mathcal{H}'(t_m)\to 0,\quad \mathrm{as}\quad m\to \infty,
\end{equation*}
which is the desired. Note that $t_m\to\infty$ because otherwise, if there is at least one accumulation point, then at this point $\ub$ must vanishes almost everywhere in $\R^5$. Thus in view of \eqref{conserQE} we obtain  $\ub_0=\mathbf{0}$, which is a contradiction.
\end{proof}

We are finally in a position to prove Theorem \ref{thm:critscatt}.  
 \begin{proof}[Proof of Theorem \ref{thm:critscatt}]
 By Theorem A, we already know that $\ub$ is globally defined and uniformly bounded in $\mathbf{H}_{x}^{1}$.   Fix $\epsilon>0$ and $R$ as in Lemma \ref{scatcrit}. Next, from Proposition \ref{energevac} we take $t_{m}\to \infty$ and $R_{m}\to \infty$ such that
 \begin{equation}\label{limReF}
    \lim_{m\to \infty}\int_{|x|\leq R_{m}}\sum_{k=1}^{l}|u_{k}(x,t_{m})|^{3}\;dx=0. 
\end{equation}
 Thus, choosing $m$ large enough  and $R_{m}\geq R$ we have from H\"{o}lder's inequality that 
 \begin{equation*}
     \begin{split}
         \left(\int_{|x|\leq R}\sum_{k=1}^{l}\frac{\alpha_k^{2}}{\gamma_{k}}|u_{k}(x,t_{m})|^{2}\;dx\right)^{\frac{1}{2}}&\leq C\sum_{k=1}^{l}\left(\int_{|x|\leq R}|u_{k}(x,t_{m})|^{2}\;dx\right)^{\frac{1}{2}}\\
         &\leq C\sum_{k=1}^{l}\left(\int_{|x|\leq R}1\;dx\right)^{\frac{1}{6}}\sum_{k=1}^{l}\left(\int_{|x|\leq R}|u_{k}(x,t_{m})|^{3}\;dx\right)^{\frac{1}{3}}\\
         &\leq CR^{\frac{5}{6}}\left(\int_{|x|\leq R_{m}}\sum_{k=1}^{l}|u_{k}(x,t_{m})|^{3}\;dx\right)^{\frac{1}{3}}.
     \end{split}
 \end{equation*}
 By sending $m\to \infty$ and using \eqref{limReF} we conclude that \eqref{liminfl2} holds. Hence, an application of Lemma \ref{scatcrit} completes the proof.
\end{proof}

\section*{Acknowledgement}
A.P. is partially supported by CNPq/Brazil grant 303762/2019-5 and FAPESP/Brazil grant 2019/02512-5.


\bibliographystyle{abbrv}
\bibliography{sample}

\end{document}